\documentclass[conference]{IEEEtran}
\IEEEoverridecommandlockouts
 \IEEEpubid{0000--0000/00\$00.00˜\copyright˜2023 IEEE}
\usepackage{cite}
\usepackage{amsmath,amssymb,amsfonts}
\usepackage{algorithmic}
\usepackage{graphicx}
\usepackage{textcomp}
\usepackage{xcolor}
\def\BibTeX{{\rm B\kern-.05em{\sc i\kern-.025em b}\kern-.08em
    T\kern-.1667em\lower.7ex\hbox{E}\kern-.125emX}}

\usepackage{multicol}
\usepackage{listings}

\definecolor{dkgreen}{rgb}{0,0.6,0}
\definecolor{gray}{rgb}{0.5,0.5,0.5}
\definecolor{mauve}{rgb}{0.20 , 0.40, 1.0}

\lstdefinelanguage{Maxima}
{morekeywords={allbut, and, block, break, buildq, do, else, elseif, error, errcatch, for,   get, go, if, in, is, lambda, local, 
		matchdeclare, new, not, or, step, then, thru, tellsimp, tellsimpafter, unless, while, return},
	sensitive=false,
	comment=[s]{/*}{*/},
	morestring=[b]"
}

\newcommand{\wrighta}[3]{\ensuremath{ W \left(\left. #1,  #2 \right|  #3 \right) }}

\newcommand{\hsym}[2]{  \ensuremath{ \; _{#1}F_{#2} }  }

\newcommand{\ai}{ \ensuremath{ \operatorname{Ai} } }
\newcommand{\erf}[1]{\ensuremath{ \mathrm{erf} \left( #1 \right)   }}
\newcommand{\erfc}[1]{ \ensuremath{ \operatorname{erfc}\left( #1\right)  } }

\newcommand{\fclass}[2]{\ensuremath{  \mathbb{#1}^{\, #2} }}

\newtheorem{theorem}{Theorem}
\newtheorem{proposition}{Proposition}
\newtheorem{proof}{Proof}

\begin{document}

\title{The Wright function -- hypergeometric representation and symbolical evaluation\\
\thanks{Funded by the European Union Horizon Europe program}
}

\author{\IEEEauthorblockN{  Dimiter Prodanov\textsuperscript{1,2}}
	
\IEEEauthorblockA{1 -- \textit{EHS and NERF}, \\
	\textit{Interuniversity Microelectronics Centre (Imec)}\\
	Leuven, Belgium \\
	ORCID 0000-0001-8694-0535
}
\IEEEauthorblockA{2 -- \textit{ITSDP}, \\
\textit{Institute for Information and Communication Technologies (IICT), Bulgarian Academy of Sciences }\\
Sofia, Bulgaria
}

}

\maketitle
\begin{abstract}
The Wright function, which arises in the theory of the space-time fractional diffusion equation, is an interesting mathematical object which has diverse connections with other special and elementary functions. The Wright function provides a unified treatment of several classes of special functions, such as the Gaussian, Airy, Bessel, and Error functions,  etc. 
The manuscript demonstrates an algorithm for symbolical representation in terms of finite sums of hypergeometric (HG) functions and polynomials. The HG functions are then represented by known elementary or other special functions, wherever possible.  The algorithm is programmed in the open-source computer algebra system Maxima and can be used to for testing numerical algorithms for the evaluation of the Wright function. 
\end{abstract}

\begin{IEEEkeywords}
Wright function, hypergeometric function, Bessel function, Error function, Airy function, Gaussian function, Whittaker function
\end{IEEEkeywords}

\section{Introduction}\label{sec:intro}
The Wright function arises in the theory of the space-time fractional diffusion equation with the temporal Caputo derivative \cite{Gorenflo2000} and the fractional cable equation used in the modeling of apical dendrites of neurons\cite{Vitali2018}.
The eponymous function provides a unified  treatment of several classes of special functions, 
notably the Bessel functions,  the probability integral function $\operatorname{erf}$,  the Airy function \ai, and the Whittaker function, among others. 
It is also related to the derivatives of the Gaussian and the Airy functions. 
The function  was originally defined by the infinite series \cite{Wright1935}:
\[
W\left(a,b \middle| \ z\right) := \sum\limits_{k=0}^{\infty} \frac{z^k}{k! \, \Gamma (a k + b)}, \quad z, b \in \mathbb{C}, \quad a>-1
\]
where $\Gamma$ denotes the Euler's gamma function.
The function is classified into Wright function of the 1\textsuperscript{st} type if $a \geq 0$ and the Wright function of 2\textsuperscript{nd} type for $-1<a<0$ \cite{Mainardi2010a,Mainardi2020}.
The Wright function is closed with regards to differentiation
\begin{equation}\label{eq:wrightdiff}
	\frac{d}{dz} \wrighta{a}{b}{z} =  \wrighta{a}{a+ b}{z}
\end{equation}
and integration
\begin{equation}\label{eq:wrightint}
	\int \wrighta{a}{b}{z} dz = \wrighta{a}{b- a}{z}   + C
\end{equation}
Formulas for positive values of $a$ have been published in \cite{Miller1995,Gorenflo1999} and  have been derived through the Meijer G function representation. 
Recently, Apelblat and Gonzales-Santander \cite{Apelblat2021} have tabulated representations in terms of hypergeometric functions for many $a$ and $b$ combinations.
The present article extends the results of \cite{Apelblat2021} for the cases wherever $a<0$ and $b>1$ and also demonstrates how the domain of $a$ can be extended into the negative integers under certain conditions by explicitly constructing polynomial representations. 

\section{The canonical complex integral representation}\label{sec:intrep}

The canonical complex line integral representation of the Wright function is given by the  integral
\[
\wrighta{a}{b}{z}  = \frac{1}{2 \pi i} \int_{Ha^{-}} \frac{e^{\xi + z \xi^{-a}}  }{\xi^{b}} d \xi, \quad z \in \fclass{C}{}
\]
along a Hankel contour, which surrounds the negative real semi-axis (Fig. \ref{fig:hankel}). 
\begin{figure}[h!]
	\centerline{\includegraphics[width=0.5\columnwidth]{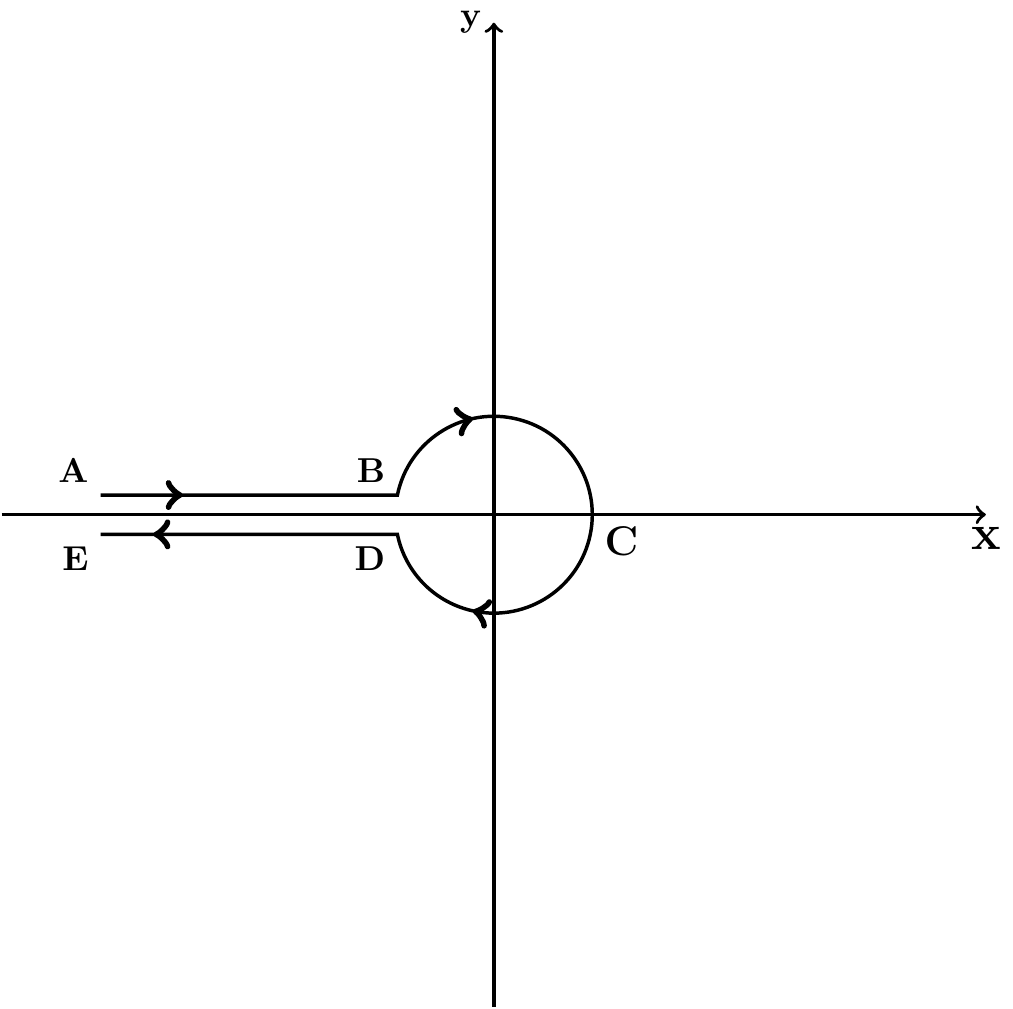}}
	\caption{Partition of the Hankel contour}
	\label{fig:hankel}
\end{figure}

\section{Polynomial reduction}\label{sec:poly}
For integral values of $b$ and $a $ the path of integration closes around the origin \textit{O}
\[
\wrighta{a}{b}{z} = \frac{1}{2 \pi i} \oint_{O} \frac{e^{\xi + z \xi^{-a}}  }{\xi^{b}} d \xi
\]
and can be used to extend the domain of the function into the negative integer parameters. 
In particular, let us consider the case when \textit{a} is a negative integer and denote it by $-n$.
Trivially, if \textit{b} is a negative integer, say $b=-m$, then the above integral vanishes and
$
\wrighta{-n}{-n}{z}=0
$.

In contrast, if $n=-a$ and $b=m, \quad m, n \in \fclass{N}{}$ then
\begin{multline*}
	\wrighta{-n}{m}{z} = \operatorname{Res} \left( ker (\xi), \xi=0\right) = \\ \frac{1}{\Gamma(m)} \left\lbrace \left( \frac{d}{d \xi}\right)^{m-1}  e^{\xi + z \xi^{n}} \right\rbrace \Bigg|_{\xi=0}
\end{multline*}
Therefore, we can conclude that \wrighta{-n}{m}{z} is a polynomial in $z$.
This is a novel result, which was not anticipated by Wright and allows for the extension of the domain of the parameters of the function. 
This polynomial can be computed explicitly by application of Fa\'a di Bruno's formula using the complete exponential Bell polynomials. 
For the natural numbers \textit{n} and \textit{m}:
\begin{equation}\label{eq:det}
	\wrighta{-n}{m}{z} =  \frac{1}{\Gamma(m)} B_{m-1}\left(g^\prime (0), \ldots ,  g^{(m-1)} (0) \right)
\end{equation}
where $ g(\xi)=\xi +z \xi^n$ is the exponent of the kernel and
\begin{multline*}
B_{m}\left(g^\prime (0), \ldots ,  g^{(m)} (0) \right) = \\
\left| 
\begin{array}{llcr}
	\binom{m-1}{0} g^\prime (0) & \binom{m-2}{1} g^{\prime\prime} (0)& \ldots & \binom{m-1}{m-1}  g^{(m)} (0) \\
	-1 & \binom{m-2}{1} g^{\prime} (0) &\ldots & \binom{m-2}{m-2}  g^{(m-1)} (0) \\
	0 & -1 &\ldots & \binom{m-3}{m-3}  g^{(m-3)} (0) \\
	\multicolumn{4}{c}{\ldots} \\
	0 & \ldots  & -1  & \binom{0}{0} g^\prime (0) 
\end{array} \right|
\end{multline*}

These expressions can be tabulated for some cases as follows.
For $m=1$ :
\[1,z+1\]
$m=2$ :
\[\frac{1}{2},\frac{2 z+1}{2},\frac{{{\left( z+1\right) }^{2}}}{2}\]
$m=3$ :
\[\frac{1}{6},\frac{6 z+1}{6},\frac{6 z+1}{6},\frac{{{\left( z+1\right) }^{3}}}{6}\]
$m=4$
\[\frac{1}{24},\frac{24 z+1}{24},\frac{24 z+1}{24},\frac{12 {{z}^{2}}+12 z+1}{24},\frac{{{\left( z+1\right) }^{4}}}{24}\]
$m=5$ :
\[\frac{1}{120},\frac{120 z+1}{120},\frac{120 z+1}{120},\frac{60 z+1}{120},\frac{60 {{z}^{2}}+20 z+1}{120},\frac{{{\left( z+1\right) }^{5}}}{120}\]

It should noted that the resulting matrix is a band matrix since already $ g^{\prime\prime} (0)=0$.
\section{Hypergeometric representation}\label{sec:hgrep}
Wherever the $a$ parameter is rational the Wright function can be represented by a finite sum of hypergeometric functions.
The Generalized Hypergeometric (GHG) function $pFq$ will be used under the notation
\begin{multline}
\hsym{p}{q}\left(a_0 \ldots a_{p-1}; b_0 \ldots b_{q-1} \Big| z \right) :=	\sum_{r=0}^{\infty} \frac{z^r}{r!} \frac{\prod_{j=0}^{p-1}\
(a_j)_r}{ \prod_{j=0}^{q-1}(b_j)_r} 
\end{multline}
where $(a)_r$ and $(b)_r$ will denote rising factorials and $(a)_0=1$, which assumes the normalization
$\hsym{p}{q}\left( \sim; \sim \Big| 0 \right) =1 $. 
By convention, equal parameters in the numerator and denominator will cancel out.
For positive, rational $a$ one could obtain the representation \cite{Apelblat2021}:
\begin{theorem}[First HG Representation]
	Suppose that $a=n/m >0$, where \textit{n} and \textit{m} are co-prime. Then
	\begin{multline}
		\wrighta{a=n/m}{b}{z} = \sum_{r=0}^{m-1} \frac{z^r}{r!\; \Gamma\left( b + a r \right) } \\ \hsym{1}{m+n}\left(1 ; b_0 \ldots b_{n-1}, c_0 \ldots c_{m-1} \Big| \frac{z^m}{n^n m^m} \right), 
	\end{multline}
	where
	\[
	b_j= r/m +(b+j)/n, \quad c_j= (r+1+j)/m
	\]
\end{theorem}
	The proof follows \cite{Miller1995} and is given as a staring point for the proof of the Second Representation Theorem. 
\begin{proof}
	Starting from
	\begin{multline*}
	\wrighta{a=n/m }{b}{z} = \sum_{n=0}^{\infty} \frac{z^n}{n! \ \Gamma (a n +b)} = \\
	\sum_{q=0}^{m-1}\sum_{p \geq q/m }^{\infty}  \frac{z^{ m p -q}}{\Gamma( m p -q +1) \ \Gamma (a ( m p -q) +b)}  
	\end{multline*}
	since the integer \textit{n} can be partitioned as $n= m p -q$, where  $q=0,  \ldots m-1$.
	After some algebra we obtain
	\begin{multline*}
	\wrighta{n/m}{b}{z} = \\
	  \frac{1}{\Gamma (b)}+ \sum_{r=1}^{m} z^r \sum_{p=0}^{\infty} \frac{z^{m p}}{ \Gamma ( ap + r a + b) \Gamma ( m p + r+1) }.
	\end{multline*}
	Observe that for $p=0$ the inner series coefficient is
	\[
	C_r= \frac{1}{\Gamma ( ra+ b) \Gamma (r+1) },
	\] 
	which serves a normalization factor.              
	Therefore, the series transforms as
	\begin{multline}\label{eq:wrightexp}
	\wrighta{n/m}{b}{z} =\\ 
	\sum_{r=0}^{m} \frac{z^r}{C_r} \cdot 
	\sum_{p=0}^{\infty}  \frac{ C_r}{\Gamma \left(n (p+r/m + b/n) \right)  }\frac{z^{m p}}{\Gamma \left(m (p + ( r+ 1)/m ) \right) }
\end{multline}
	Further, use Prop. \ref{prop:gammamult} to obtain
	\begin{multline*}
	\Gamma (b + a j ) = \Gamma \left( a \left( j + b/a \right)  \right) = \\
	 \Gamma (b) \left( \frac{b}{a} \right)_{j} \left( \frac{b+1}{a} \right)_{j} \ldots \left( \frac{b+a-1}{a} \right)_{j} a^{aj}.
	\end{multline*}
	From eq. \ref{eq:wrightexp} we read off
	\[
	b_0= \frac{r}{m} + \frac{b}{n}, \quad c_0 = \frac{r+1}{m}
	\]
	with increments $1/n$ and $1/m$, respectively. 
	Therefore, finally,
	\begin{multline*}
	\wrighta{a}{b}{z} = \sum_{r=0}^{m-1} \frac{z^r}{r!\; \Gamma\left( b + a r \right) } \cdot \\ \hsym{1}{m+n}\left(1 ; b_0 \ldots b_{n-1}, c_0 \ldots c_{m-1} \Big| \frac{z^m}{n^n m^m} \right).
	\end{multline*}
	Observe that for $r=m-1$ $c_1=1$ therefore, the GHG functions reduce to \hsym{0}{m+n-1}.
\end{proof}

The formula for a negative rational $a <0$ needs some more work. 
Suppose first that $b<1$. 
Let 
\begin{multline*} 
	\wrighta{-n/m}{b}{z} = \frac{1}{\Gamma (b)} + \\
	\sum_{r=1}^{m} \frac{z^r}{C_r} \sum_{p=0}^{\infty}  \frac{C_r}{\Gamma \left( -np-rn/m + b \right)  }\frac{z^{m p}}{\Gamma \left(m p + r+ 1  \right) }
\end{multline*}
First we use the Gamma reflection formula to obtain
\begin{equation}
\frac{1}{\Gamma \left( -np-rn/m + b \right) }= \frac{{{\left( -1\right) }^{n p}} \Gamma\left( \frac{n r}{m}+n p-b+1\right) }{\Gamma\left( b-\frac{n r}{m}\right) \Gamma\left( \frac{n r}{m}-b+1\right) }
\end{equation}
Therefore, 
\begin{multline*}
	\wrighta{-n/m}{b}{z} = \frac{1}{\Gamma (b)} + \\
	\sum_{r=1}^{m} \frac{z^r}{C_r} \sum_{p=0}^{\infty}   \frac{ (-1)^{n p} C_r \Gamma\left( \frac{n r}{m}+n p-b+1\right)  z^{m p}}{\Gamma\left( \frac{n r}{m}-b+1\right) \Gamma \left(m (p + ( r+ 1)/m ) \right) }
\end{multline*}
Finally, we read off the parameters
\[
b_j^\prime= 1-  (-r/m +(b+j)/n) = 1+ r/m - (b+j)/n .
\] 
Then we can formulate the following
\begin{theorem}[Second HG Representation]\label{th:rep2}
	For $ b \leq 1$ and $n \leq m$ non-negative co-prime integers
	\begin{multline}
		\wrighta{a=-n/m}{b}{z} = \sum_{r=0}^{m-1} \frac{z^r}{r!\; \Gamma\left( b + a r \right) } \\ \hsym{n+1}{m}\left(1 , b_0^\prime \ldots b_{n-1}^\prime; c_0 \ldots c_{m-1} \Big| \frac{(-)^n z^m}{n^n m^m} \right), 
	\end{multline}
	where 
	\[
	b_j^\prime=  1+ r/m - (b+j)/n  \quad c_j= (r+1+j)/m
	\]
\end{theorem}
For $b\geq1$ a polynomial part $P$ is also added as follows.

\begin{theorem}[Third HG representation]\label{th:rep3}
	Suppose that \textit{a} and  \textit{b} are rational parameters and $b \geq 1 $  and $|a| \leq 1$.
	Define the polynomial $P_b(-a, z)$ by the integral recursion
	\begin{equation}
		P_b(-a, z)= \int_{0}^{z} P_{b-a}( -a, x) dx + c_{b-1},
	\end{equation}
	where $c_{b-1}=1/(b-1)!$ if $b$ is an integer and $0$ otherwise. 
	Furthermore, define $P_0 (z, -a) :=1$ and for $b < 0 $ assign $P_b(z, -a) :=0$ identically.
	Then
	\begin{multline}
		\wrighta{a=-n/m}{b}{z} = \sum_{r=0}^{m-1} \frac{z^r}{r!\; \Gamma\left( b + a r \right)} \cdot \\ 
		\hsym{n+1}{m}\left(1 , b_0^\prime \ldots b_{n-1}^\prime; c_0 \ldots c_{m-1} \Big| \frac{(-)^n z^m}{n^n m^m} \right) + P_b(-a, z)
	\end{multline}
	where \textit{m} and \textit{n} are co-prime numbers. 
\end{theorem}
\begin{proof}
	First we prove that the arc integral results in a polynomial in \textit{z}.
	Suppose that $b \geq 1 $ is rational and $-a=n/m$ as before. 
	Consider the arc BCD. 
	We change variables as $\xi= \epsilon \eta^m $.
	Then the integral becomes
	\[
	I= m \; \oint_{BCD}\frac {d\eta \; {\epsilon}^{1-b}} {{\eta}^{ (b -1) m+1}}  {{e}^{{{\epsilon}^{\frac{n}{m}}} {{\eta}^{n}} z+\epsilon {{\eta}^{m}}}}
	\]
	Development  of the kernel in infinite series results in 
	\[
	ker= m\,\frac{d\eta \; {\epsilon}^{1-b}}  {{\eta}^{ (b -1) m+1}}  \sum_{j=0}^{\infty }{\left. \sum_{i=0}^{j}{\left. \frac{{{\epsilon}^{\frac{i m+\left( j-i\right)  n}{m}}} {{\eta}^{i m+\left( j-i\right)  n}} {{z}^{j-i}}}{i{!} \left( j-i\right) {!}}\right.}\right.}
	\]
	The scale-invariant part of the series is given by the members $k_{j}$ for which
	$
	{{\epsilon}^{\frac{i m+\left( j-i\right)  n}{m}-b+1}}=1
	$.
	This is given by the constraint
	\[
	i=\frac{\left( b-1\right)  m-j n}{m-n}
	\]
	Therefore,
	\[
	k_{j}=\frac{  d\eta \; m\;   {z}^{\frac{\left( j-b+1\right)  n}{n-m}}}{\eta \left( \frac{\left( b-1\right)  n-j m}{n-m}\right) {!} \left( \frac{\left( j-b+1\right)  n}{n-m}\right) {!}}\]
	Changing again variables to $\eta= e^{i \varphi/m}$ results in the integral
	\begin{multline*}
		c_j=\frac{1}{2 \pi i}
		\frac{i  {{z}^{\frac{\left( j-b+1\right)  m}{m-n}}}   }{\left( \frac{\left( b-1\right)  m-j n}{m-n}\right)  {!} \left( \frac{\left( j-b+1\right)  m}{m-n}\right)  {!}} \cdot \\
		 { \int_{-\pi}^{\pi}  {{e}^{\frac{i \left( b n-1\right)  \varphi}{n}+\frac{i \varphi}{n}-i b \varphi}} d \varphi}  =  \frac{  {{z}^{\frac{\left( j-b+1\right)  m}{m-n}}}   }{\left( \frac{\left( b-1\right)  m-j n}{m-n}\right)  {!} \left( \frac{\left( j-b+1\right)  m}{m-n}\right)  {!}} 
	\end{multline*}
	Furthermore, the valid indices are given by the set
	\[
	j: \quad \left(  \frac{\left( j-b+1\right)  m}{m-n} \in \fclass{N}{}  \right)  \cup \left( \frac{\left( b-1\right)  m-j n}{m-n} \in \fclass{N}{} \right)
	\]
	Equivalently, in the a-notation
	\begin{equation}\label{eq:coeffa}
	c_j=\frac{{z}^{\frac{j-b+1}{1-a}}}{ \left( \frac{j-b+1}{1-a}\right)  {!} \left( \frac{-a j+b-1}{1-a}\right)  {!} }
	\end{equation}
	Therefore, $a <1$ must hold for $c_j$ not to vanish. 
	
	On the other hand,
	\[
	b-1 \leq j \leq (b-1) /a \quad \cup \quad j \in \fclass{N}{}
	\]
	which is a finite set. Therefore,  for a rational \textit{b} the integral \textit{I} is a polynomial in \textit{z}.
	
	To derive the polynomial recursion we proceed as follows. 
	By Prop. \ref{prop:inteq1}
	\begin{equation}\label{eq:intpol1}
		\wrighta{-a}{b}{z} = \int_{0}^{z} \wrighta{-a}{b-a}{z} + \frac{1}{\Gamma(b)}
	\end{equation}
	so that the equation defines a recursion relationship.  
	
	Observe that for $j=b-1$ the coefficient becomes
	\[
	c_{b-1}= \frac{1}{ (b-1)!}
	\] 
	Therefore, for non-integer $b$ there are no constant monomials.
	Furthermore, consider the monomial $c_j$ as a function of $b$.
	Differentiating eq. \ref{eq:coeffa} one obtains the recursion
	\begin{multline*}
	\frac{d}{dz} c_j (b)=\frac{{z}^{\frac{j-b+1}{1-a} -1}}{ \left( \frac{j-b+1}{1-a} -1\right)  {!} \left( \frac{-a j+b-1}{1-a}\right)  {!} } =  \\
	\frac{{{z}^{\frac{j- (b-a)}{1-a}}}}{\left( \frac{j-(b-a)}{1-a}\right)  {!} \left( \frac{a j-b+1}{a-1}\right)  {!}  } =  c_{j-1} (b-a),
	\end{multline*}
	which is also consistent with the integral eq. \ref{eq:intpol1}.
	Therefore, the polynomial $P_b(-a, z)$ should obey the above recursion. 
	The second argument of the Wright function mutates and therefore it is convenient that it indexes the polynomial.
\end{proof}
For integer values of $a$, that is, when $m=1$ Th. \ref{th:rep3} corresponds with the polynomial representation since the  hypergeometric sum disappears.

It should be remarked that The Second and Third Representation theorems could not be traced to available literature. In particular, they are not stated in \cite{Gorenflo1999,Mainardi2001,Apelblat2021,Mainardi2022,Mainardi2020,Povstenko2021}.

\section{Implementation and Examples}\label{sec:examples}
The theorems have been implemented as functionality in the Computer Algebra System (CAS) Maxima.
The main functionality is presented in Listing \ref{sec:code} as the function \textrm{wrightsimp}.
Maxima has a hypergeometric reduction functionality by the function \textrm{hgfred}, which allows for seamless integration of so-presented theory \cite{Avgoustis1977}. 
This allows for calculation of potentially a very large number of derivative formulas.
It is possible to compute arbitrarily long decomposition but in view of the limited column space only denominators of up to 3 will be discussed. 
An illustrative example of Th. \ref{th:rep2} can be readily computed for $a=1/3$:
\begin{multline}
\wrighta{1/3}{b}{z}=\frac{ \hsym{0}{ 3}  \left( - ; b+\frac{2}{3},\frac{4}{3},\frac{5}{3}; \frac{ {z}^{3}}{27}\right)  {z}^{2}}{2  {\Gamma }\left( b+\frac{2}{3}\right) }+ \\ 
\frac{  \hsym{0}{ 3}  \left( - ,b+\frac{1}{3},\frac{2}{3},\frac{4}{3}; \frac{{z}^{3}}{27}\right)  z}{\Gamma \left( b+\frac{1}{3}\right) }+ 
\frac{ \hsym{0}{ 3}  \left( - ; b,\frac{1}{3},\frac{2}{3};\frac{{z}^{3}}{27}\right) }{ \Gamma (b)}
\end{multline}

\subsection{Representations for $a=-1/4$}\label{sec:rep4}
The Mainardi function \cite{Mainardi2010}
\[
M_a(z) =\wrighta{-a}{1-a}{-z}
\]
can be represented as the sum
\begin{multline}
M_{1/4}(z)= \frac{ \hsym{0}{ 2} \left( -; \frac{5}{4},\frac{3}{2}  ;-\frac{{{z}^{4}}}{256}\right)  {{z}^{2}}}{2  {\Gamma }\left( \frac{1}{4}\right) } \\
-\frac{\hsym{0}{ 2} \left( -;  \frac{3}{4}, \frac{5}{4} ;  -\frac{{{z}^{4}}}{256}\right)  z}{\sqrt{\pi }}+\frac{\hsym{0}{ 2}  \left( -;   \frac{1}{2},\frac{3}{4} ;  -\frac{{{z}^{4}}}{256}\right) }{ {\Gamma }\left( \frac{3}{4}\right) }
\end{multline}
For $b=3/4$
\[
\wrighta{-1/4}{3/4}{-z}=\frac{\sqrt{z}}{\sqrt{2} \pi } K_{1/4} \left( \frac{{z}^{2}}{8}\right)   {{e}^{-\frac{{z}^{2}}{8}}}
\]

\subsection{Representations for $a=-1/3$}\label{sec:rep3}
\begin{figure}[h]
	\centering
	\includegraphics[width=1.\columnwidth]{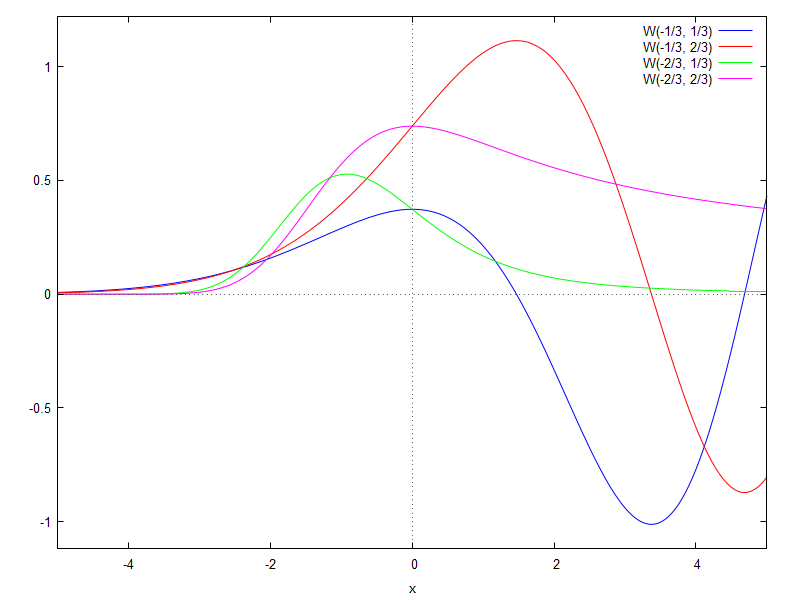}
	\caption{Plots of \wrighta{a}{b}{x} for $a=-1/3$ and $a=-2/3$}
	\label{fig:wright3}
\end{figure}
The general formula for $b \leq 1$ reads
\begin{multline}\label{eq:w3}
	\wrighta{-1/3}{b}{z}=\frac{\hsym{1}{2}
	\left(  \frac{5}{3}-b ; \frac{4}{3} , \frac{5}{3}  ;-\frac{{{z}^{3}}}{27}\right)  {{z}^{2}}}{2 {\Gamma}\left( b-\frac{2}{3}\right) }+ \\
\frac{\hsym{1}{2} \left(  \frac{4}{3}-b  ;  \frac{2}{3}, \frac{4}{3}  ;-\frac{{{z}^{3}}}{27}\right)  z}{{\Gamma}\left( b-\frac{1}{3}\right) }+\frac{\hsym{1}{2} \left(  1-b ; \frac{1}{3},\frac{2}{3} \operatorname{,}-\frac{{{z}^{3}}}{27}\right) }{{\Gamma}(b)}
\end{multline}
For $b=1/3$ the equation reduces to 
\begin{multline}
\wrighta{-1/3}{1/3}{-z}= 
-\frac{\left( I_{2/3}\left( \frac{2 {{z}^{\frac{3}{2}}}}{{{3}^{\frac{3}{2}}}}\right) -I_{ -2/3}\left( \frac{2 {{z}^{\frac{3}{2}}}}{{{3}^{\frac{3}{2}}}}\right) \right)  z}{3}
= \\
\frac{z}{\sqrt{3}\pi} K_{2/3} \left(\frac{2 {{z}^{\frac{3}{2}}}}{{{3}^{\frac{3}{2}}}} \right) = - \sqrt[3]{3} \ai^\prime \left(\frac{z}{\sqrt[3]{3}} \right) 
\end{multline}

Eq. \ref{eq:w3} simplifies as expected for $b=2/3$ to the Airy \ai function, which can be represented as a weighted sum of Bessel J or I functions, respectively.
For $z>0$
\begin{multline}
\wrighta{-1/3}{2/3}{-z}= \\
\frac{I_{ -1/3}\left(  \frac{2 {{z}^{\frac{3}{2}}}}{{{3}^{\frac{3}{2}}}}\right)  \sqrt{z}}{\sqrt{3}}-\frac{I_{ 1/3}\left(  \frac{2 {{z}^{\frac{3}{2}}}}{{{3}^{\frac{3}{2}}}}\right)  \sqrt{z}}{\sqrt{3}} = \\
\frac{K_{ 1/3}\left(  \frac{2 {{z}^{\frac{3}{2}}}}{{{3}^{\frac{3}{2}}}}\right)  \sqrt{z}}{\pi} = \sqrt[3]{3^2} \ai \left( z/ \sqrt[3]{3}\right)
\end{multline}

\subsection{Representations for $a=-1/2$}\label{sec:rep2}
For $b\leq 1$
\begin{multline}
	\wrighta{-1/2}{b}{z}=  \frac{ \hsym{1}{ 1} \left( \frac{3}{2}-b  ; \frac{3}{2} ; -\frac{{{z}^{2}}}{4}\right)  z}{ {\Gamma}\left( b-\frac{1}{2}\right) }+ \\
	\frac{\hsym{1}{ 1}  \left(   1-b ; \frac{1}{2} ;  -\frac{{{z}^{2}}}{4}\right) }{ \Gamma(b)} 
\end{multline}
In particular, for $b=1$ the above equation reduces to
\[
\wrighta{-1/2}{1}{z}= 1 + \erf{\frac{z}{2}} =\erfc{-\frac{z}{2}}
\] 
in accordance with the polynomial reduction.

The Gaussian derivatives can be represented as
\begin{equation}
\wrighta{- \frac{1}{2} }{\frac{1-n}{2}}{ z}=\left( \frac{d}{d z}\right)^n \frac{e^{-z^2/4}}{\sqrt{\pi}}  
\end{equation}
Plots are presented in Fig. \ref{fig:wright2}.
The anti-derivatives of the Gaussian kernel can be computed in a similar way using Th. \ref{th:rep3}.
For $b=7/2$
\begin{multline}
	\wrighta{-1/2}{7/2}{z}=\left( \frac{{{z}^{4}}}{60 \sqrt{\pi }}+\frac{3 {{z}^{2}}}{10 \sqrt{\pi }} +\frac{8}{15 \sqrt{\pi }}\right)  {{e}^{-\frac{{z}^{2}}{4}}} \\
	+\left( 1+ \erf {  \frac{z}{2}} \right)  \left( \frac{{z}^{5}}{120}+\frac{{z}^{3}}{6}+\frac{z}{2}\right)
\end{multline}

\begin{figure}[h!]
	\centering
	\includegraphics[width=1.\columnwidth]{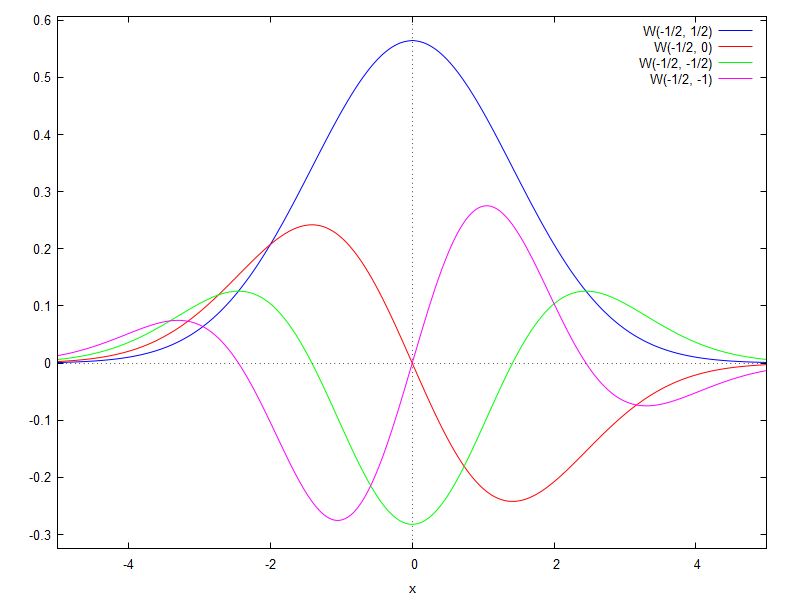}
	\caption{Plots of \wrighta{a}{b}{x} for $a=-1/2$}
	\label{fig:wright2}
\end{figure}

\subsection{Representations for $a=-2/3$}\label{sec:rep23}

The Mainardi function  for $a=2/3$ can be represented as the difference of two exponentially-weighted Bessel K functions on the entire real line:
\begin{multline}
	\wrighta{-2/3}{1/3}{-z}= \\
	\frac{K_{2/3 } \left( -\frac{2 {{z}^{3}}}{27}\right)  {{z}^{2}} {{e}^{-\frac{2 {{z}^{3}}}{27}}}}{{{3}^{\frac{3}{2}}} \pi }-\frac{K_{1/3} \left(  -\frac{2 {{z}^{3}}}{27}\right)  {{z}^{2}} {{e}^{-\frac{2 {{z}^{3}}}{27}}}}{{{3}^{\frac{3}{2}}} \pi } 
\end{multline}
On the other hand,
\[
\wrighta{-2/3}{1/3}{z}=-\frac{{{e}^{\frac{2 {{z}^{3}}}{27}}}\, \left( 3 \ai^\prime\left( \frac{{{z}^{2}}}{{{3}^{\frac{4}{3}}}}\right) +{\sqrt[3]{3} } z \ai\left( \frac{{{z}^{2}}}{{{3}^{\frac{4}{3}}}}\right) \right) }{{{3}^{\frac{2}{3}}}}
\]
in terms of the Airy function and its derivative. Plots are presented in Fig. \ref{fig:wright3}.

For $b=2/3$
\begin{multline}
\wrighta{-2/3}{2/3}{z}= \\
\frac{ K_{1/3 }\left(\frac{2 {{z}^{3}}}{27}\right)  z {{e}^{\frac{2 {{z}^{3}}}{27}}}}{\sqrt{3} \pi } 
=\sqrt[3]{9}\, {{e}^{\frac{2 {{z}^{3}}}{27}}}  \ai \left( \frac{{{z}^{2}}}{{{3}^{\frac{4}{3}}}}\right) 
\end{multline}
\subsection{Representations for $a=-1$}\label{sec:rep1}
This formula was recently derived in \cite{Apelblat2021} and is not anticipated in the previous literature since the parameter domain is customarily restricted to $a>-1$.
\begin{equation}
	\wrighta{-1}{b}{z}= \frac{1}{\Gamma(b)} \hsym{1}{0}(1-b;-; z) =  \frac{(z+1)^{b-1}}{\Gamma(b)}
\end{equation}

\section{Discussion}\label{sec:disc}
The main application of the present method is to provide ground truth for purely numerical algorithms for the evaluation of the Wright function. 
Such algorithms are a subject of continuous development \cite{Aceto2022, Luchko2008, Luchko2010}.
An advantage of the Wright function is that like the Fox H and Meijer G functions it is closed with regards to differentiation and integration. 
One can also envision application in definite integration to be incorporated into different CAS integration routines.

\section*{Acknowledgment}

The present work was funded by the European Union's Horizon Europe program under grant agreement VIBraTE 101086815.

\appendix
\subsection{Supporting results}\label{sec:support}
\begin{proposition}\label{prop:gammamult}
	For non negative integers $n, m$
	\[
	\Gamma (m n +b ) = \Gamma (m b)\; m^{m n} \prod_{j=0}^{m-1} \left(\frac{j}{m} +b \right)_n
	\]
\end{proposition}
\begin{proof}
	Using the  Gauss-Legendre multiplication formula for the Gamma function
	\[
	\Gamma(m x) = \frac{m^{m x-1/2}} {(2 \pi)^{(m-1)/2}} \prod_{k=0}^{m-1} \Gamma\left(x + \frac{k}{m} \right) 
	\]
	Thus for non negative integer $n$ the formula can be expressed by a product of increasing factorials:
	\begin{multline*}
		\frac{\Gamma (m n +b )}{\Gamma (m b)} =m^{m n} \prod_{j=0}^{m-1} \frac{\Gamma\left(n + \frac{j}{m} +b \right)}{\Gamma \left(\frac{j}{m} +b \right) } = \\
		m^{m n} \prod_{j=0}^{m-1} \left(\frac{j}{m} +b \right)_n
	\end{multline*}
\end{proof}
We use the  closure under differentiation (eq. \ref{eq:wrightdiff}) to formulate
\begin{proposition}[Integral equation]\label{prop:inteq1}
	\begin{equation}\label{eq:wrightint}
		\wrighta{a}{b}{z} = \int_{0}^{z} \wrighta{a}{a+ b}{z} + \frac{1}{\Gamma(b)}
	\end{equation}
	since by definition $\wrighta{a}{b}{0}=1/ \Gamma(b)$. 
\end{proposition}
\subsection{HG representation Maxima code}\label{sec:code}
\begin{lstlisting}
wrightsimp(wexpr):=block(
	 [ l, r, ret:wexpr, a, b, x,  z, mm, nn, lsta:[], lstb:[],  ss:0, j, 
	 listarith:true, gamma_expand:true, assume_pos:true ],
	 if freeof (wright, wexpr) then return (wexpr),
	 if inop(wexpr)="+" then 
	 	ret:map(lambda([z],  wrightsimp(z)), wexpr)
	 elseif inop(wexpr)='wright then (	
		 [a, b, x]:args( wexpr),
		 if a=0 then
		 	if  integerp(b) and  b <=0  then return(0) 
		 else 
		 	ret:exp(x)/gamma(b),
		 if rationalp(a) and a>0 then (
			 nn: num(a), mm:denom (a),
			 for r:0 thru mm-1 do ( 
				 lsta: makelist(r/mm +(b+j)/nn, j, 0, nn-1),
				 lstb: makelist((r+1+j)/mm, j, 0, mm-1),
				 ss: x^r/r! /gamma(a*r+b)* hgfred([1], append(lsta, lstb), x^mm/ mm^mm/ nn^nn)+ss
			 ),
			 ret:ss
		 ),
		 if rationalp(a) and a<0 then (
			 nn:abs(num(a)), mm:denom (a),
			 for r:0 thru mm-1 do ( 
				 /* we use the Gamma reflection formula */
				 lsta: makelist(1+r/mm -(b+j)/nn, j, 0, nn-1),
				 lstb: makelist((r+1+j)/mm, j, 0, mm-1),
				 if  not integerp(a*r+b) then
				 	ss: x^r/r! /gamma(a*r+b)* hgfred( cons(1, lsta), lstb, (-1)^nn*x^mm/ mm^mm *nn^nn)+ss
				 ),
				 if b>=1 then (
				 /* polynomial component */
				 ss:ss + mdx3(x, b, abs(a))
			 ),
			 ret:ss 
		 )		 
	) else (
	 [l, r]: partbyvar(wexpr, wright),
	 l:subst(nil=1, l),
	 ret:l * wrightsimp(r)
		 ),
	ret
);
	 
mdx3(x, n, q):=block([ s:0, xi], 
	if n<=0 then return(0),
	if n=1 then return(1) 
	else (
		if integerp(n) then s:1/gamma(n), 
		s:integrate(mdx3(xi, n-q,q) ,xi, 0, x)+s, 
		expand(s)
	)
);
	 
	 
\end{lstlisting}
\bibliographystyle{IEEEtran}
\bibliography{specfunct}

\end{document}